\documentclass[11pt]{amsart}
\usepackage{amsbsy}
\usepackage{amsmath}
\usepackage{amsfonts}
\usepackage{mathtools}
\usepackage{graphicx,epsfig,subfigure,psfrag}
\usepackage{color}
\usepackage{xcolor}
\usepackage{changes}

\begin{document}
\newcommand{\teal}{\textcolor{teal}}
\newcommand{\red}{\textcolor{red}}
\newcommand{\blue}{\textcolor{blue}}
\newcommand{\orange}{\textcolor{orange}}
\newcommand{\h}{\textcolor{orange}}
\newcommand{\pink}{\textcolor{pink}}
\newcommand{\green}{\textcolor{green}}
\newcommand{\violet}{\textcolor{violet}}
\newcommand{\rf}{\textcolor{green}{ Refer }}
\newtheorem{theorem}{Theorem}[section]
\newtheorem{proposition}[theorem]{Proposition}
\newtheorem{lemma}[theorem]{Lemma}
\newtheorem{corollary}[theorem]{Corollary}
\newtheorem{definition}[theorem]{Definition}
\newtheorem{remark}[theorem]{Remark}
\newcommand{\tex}{\textstyle}
\numberwithin{equation}{section}
\newcommand{\ren}{\mathbb{R}^N}
\newcommand{\re}{\mathbb{R}}
\newcommand{\n}{\nabla}
\newcommand{\p}{\partial}
\newcommand{\iy}{\infty}
\newcommand{\pa}{\partial}
\newcommand{\fp}{\noindent}
\newcommand{\ms}{\medskip\vskip-.1cm}
\newcommand{\mpb}{\medskip}
\newcommand{\AAA}{{ A}}
\newcommand{\BB}{{ B}}
\newcommand{\CC}{{ C}}
\newcommand{\DD}{{ D}}
\newcommand{\EE}{{ E}}
\newcommand{\FF}{{ F}}
\newcommand{\GG}{{ G}}
\newcommand{\oo}{{\mathbf \omega}}
\newcommand{\Am}{{ A}_{2m}}
\newcommand{\CCC}{{\mathbf  C}}
\newcommand{\II}{{\mathrm{Im}}\,}
\newcommand{\RR}{{\mathrm{Re}}\,}
\newcommand{\eee}{{\mathrm  e}}
\newcommand{\LL}{L^2_\rho(\ren)}
\newcommand{\LLL}{L^2_{\rho^*}(\ren)}
\renewcommand{\a}{\alpha}
\renewcommand{\b}{\beta}
\newcommand{\g}{\gamma}
\newcommand{\G}{\Gamma}
\renewcommand{\d}{\delta}
\newcommand{\D}{\Delta}
\newcommand{\e}{\epsilon}
\newcommand{\var}{\varphi}
\newcommand{\lll}{\l}
\renewcommand{\l}{\lambda}
\renewcommand{\o}{\omega}
\renewcommand{\O}{\Omega}
\newcommand{\s}{\sigma}
\renewcommand{\t}{\tau}
\renewcommand{\th}{\theta}
\newcommand{\z}{\zeta}
\newcommand{\wx}{\widetilde x}
\newcommand{\wt}{\widetilde t}
\newcommand{\noi}{\noindent}
\newcommand{\uu}{{ u}}
\newcommand{\xx}{{ x}}
\newcommand{\yy}{{ y}}
\newcommand{\zz}{{ z}}
\newcommand{\aaa}{{ a}}
\newcommand{\cc}{{ c}}
\newcommand{\jj}{{ j}}
\newcommand{\ggg}{{ g}}
\newcommand{\UU}{{ U}}
\newcommand{\YY}{{ Y}}
\newcommand{\HH}{{ H}}
\newcommand{\GGG}{{ G}}
\newcommand{\VV}{{ V}}
\newcommand{\ww}{{ w}}
\newcommand{\vv}{{ v}}
\newcommand{\hh}{{ h}}
\newcommand{\di}{{\rm div}\,}
\newcommand{\ii}{{\rm i}\,}
\def\I{{\mathbf{I}}}
\newcommand{\inA}{\quad \mbox{in} \quad \ren \times \re_+}
\newcommand{\inB}{\quad \mbox{in} \quad}
\newcommand{\inC}{\quad \mbox{in} \quad \re \times \re_+}
\newcommand{\inD}{\quad \mbox{in} \quad \re}
\newcommand{\forA}{\quad \mbox{for} \quad}
\newcommand{\whereA}{,\quad \mbox{where} \quad}
\newcommand{\asA}{\quad \mbox{as} \quad}
\newcommand{\andA}{\quad \mbox{and} \quad}
\newcommand{\withA}{,\quad \mbox{with} \quad}
\newcommand{\orA}{,\quad \mbox{or} \quad}
\newcommand{\atA}{\quad \mbox{at} \quad}
\newcommand{\onA}{\quad \mbox{on} \quad}
\newcommand{\ef}{\eqref}
\newcommand{\mc}{\mathcal}
\newcommand{\mf}{\mathfrak}

\newcommand{\ssk}{\smallskip}
\newcommand{\LongA}{\quad \Longrightarrow \quad}
\def\com#1{\fbox{\parbox{6in}{\texttt{#1}}}}
\def\N{{\mathbb N}}
\def\A{{\cal A}}
\newcommand{\de}{\,d}
\newcommand{\eps}{\epsilon}
\newcommand{\be}{\begin{equation}}
\newcommand{\ee}{\end{equation}}
\newcommand{\spt}{{\mbox spt}}
\newcommand{\ind}{{\mbox ind}}
\newcommand{\supp}{{\mbox supp}}
\newcommand{\dip}{\displaystyle}
\newcommand{\prt}{\partial}
\renewcommand{\theequation}{\thesection.\arabic{equation}}
\renewcommand{\baselinestretch}{1.1}
\newcommand{\Dm}{(-\D)^m}
\title
{ On an n-dimensional fourth-order system under a parametric condition}

\author{Pablo \'Alvarez-Caudevilla, Cristina Br\"{a}ndle and Devashish Sonowal}

\address{Universidad Carlos III de Madrid,
Av. Universidad 30, 28911-Legan\'es, Spain} \email{pacaudev@math.uc3m.es}

\address{Universidad Carlos III de Madrid,
Av. Universidad 30, 28911-Legan\'es, Spain} \email{cbrandle@math.uc3m.es}

\address{Universidad Carlos III de Madrid,
Av. Universidad 30, 28911-Legan\'es, Spain} \email{dsonowal@math.uc3m.es}

\keywords{Coupled systems, higher order operators}

\thanks{This paper has been partially supported by Ministry of Economy and Competitiveness of Spain under research project PID2019-106122GB-I00}

\subjclass{35J70, 35J47, 35K57}

\date{\today}

\parskip5pt

\begin{abstract}
We establish the existence of positive solutions for a system of coupled fourth-order partial differential equations on a bounded domain $\Omega \subset \mathbb{R}^n$
\begin{align*}
    \left\{
         \begin{array}{l} \Delta^2u_1 +\beta_1 \Delta u_1-\alpha_1 u_1=f_1({ x},u_1,u_2),\\
         \Delta^2 u_2+\beta_2\Delta u_2-\alpha_2 u_2=f_2({ x},u_1,u_2), \end{array}  \quad \quad x\in\Omega,
    \right.
\end{align*}
subject to homogeneous Navier boundary conditions, where the functions $f_1,f_2 : \Omega\times [0,\infty)\times [0,\infty) \rightarrow [0,\infty)$ are continuous, and $\alpha_1,\alpha_2,\beta_1$ and $\beta_2$ are real parameters satisfying certain constraints related to the eigenvalues of the associated Laplace operator.
\end{abstract}

\maketitle

\section{Introduction}
 \label{S1}

Fourth-order nonlinear differential equations naturally appear in models concerning physical, biological, and chemical phenomena, such as, for instance, problems of elasticity, deformation of structures, or soil settlement, see, for example,
\cite{PelTro} and \cite{Gaz2} for the exposition of several models involving higher order operators.
As it is explained in \cite{PelTro}, typically in the literature we find fourth-order ordinary differential equations of the form
\begin{equation*}
\label{first}
\left\{
\begin{array}{l}
u^{(iv)}(x)=g(x,u(x), u''(x)), \qquad 0<x<1,\\
u(0)=u(1)=u''(0)=u''(1),
\end{array}\right.
\end{equation*}
under different conditions on the function $g$. Such kind of equations are used to model the deformations of an elastic beam in equilibrium state, whose two ends are simply supported~\cite{Gupta1988},~\cite{Gupta1998}. Due to its physics applications one normally looks for the existence of positive solutions. Thus, for the particular one-dimensional case there are several papers where
such an existence is analysed; see \cite{DelPino_Manas}, \cite{Gupta1988}, \cite{Gupta1998}, \cite{Li} for further details and references therein.
Although there are numerous references dealing with this type of one dimensional problems,
little is yet known about the behaviour of the solution in higher dimensions, both for a single equation or a system. Indeed, for one single equation the work \cite{DelPino_Manas} is probably
the only one mentioning higher dimension
equations. On the other hand,
 we would like to mention the work of Wang-Yang \cite{Wang-Yang} where a system of fourth order differential equations was analysed obtaining the existence of positive solutions, however in one dimension.

In this work we consider a bounded smooth domain $\Omega\subset \mathbb{R}^N$, with $N\geq 1$ and  we generalise the one dimensional system studied in \cite{Wang-Yang} to a system with  two coupled equations of the form
$$    \left\{
         \begin{array}{l} \Delta^2u_1 =g_1({ x},,u_1,u_2,\Delta u_1),\\
         \Delta^2 u_2=g_2({ x},u_1,u_2,\Delta u_2), \end{array}
    \right.
$$
for some functions $g_1$ and $g_2$. More precisely, and in relation to the previously analysed problems in 1D,
we will be interested in discussing the existence of positive solutions $(u_1,u_2)$ to the system
\begin{equation}\label{1.1_eq_principal}
    \left\{
         \begin{array}{l} \Delta^2u_1 +\beta_1 \Delta u_1-\alpha_1 u_1=f_1({ x},u_1,u_2),\\
         \Delta^2 u_2+\beta_2\Delta u_2-\alpha_2 u_2=f_2({ x},u_1,u_2), \end{array}  \quad \quad x\in\Omega,
    \right.
\end{equation}
under the homogeneous Navier boundary conditions
\begin{equation}\label{1.1_BC}
u_1=\Delta u_1=u_2=\Delta u_2=0  \quad \text{on }\partial \Omega.
\end{equation}
We will assume that the functions $f_1,f_2   \,:\, \Omega\times [0,\infty)\times [0,\infty) \rightarrow [0,\infty)$ are continuous and $\alpha_1,\alpha_2,\beta_1,\beta_2$ are real parameters. Our aim in this article is to show that under additional conditions on the parameters (the so-called non-resonance condition, see~\cite{DelPino_Manas})  of the system and on the growth conditions of $f_1$ and $f_2$, the system has at least one positive solution (see Section~\ref{proof_theorem}).

\noindent{\sc Outline of the paper:} In Section \ref{sect:preliminaries_elements} we show some crucial and important results and properties for Green's functions that will be essential in proving the existence of solutions for problem \eqref{1.1_eq_principal}--\eqref{1.1_BC}. Those existence results will be obtained through the application of general fixed point theory showed in Section \ref{Section_Existence}. The final Section \ref{proof_theorem} is devoted to the proof of the existence result.


\section{Preliminaries: the case of a single equation}
\label{sect:preliminaries_elements}


To show the existence of solution for system \eqref{1.1_eq_principal} under the homogeneous Navier boundary conditions \eqref{1.1_BC} it seems to be convenient to study the behaviour of the problem
\begin{equation}\label{eigenvalue_problem}
\left\{\begin{array}{ll}
\Delta^2u+\beta \Delta u-\alpha u = 0 & \quad\text{in } \Omega,\\  u=\Delta u=0 & \quad \text{on }\partial \Omega,
\end{array}\right.
\end{equation}
where $\alpha$ and $\beta$ are two real parameters. As a first step, observe that this fourth order equation~\eqref{eigenvalue_problem} can be rewritten as
\begin{equation}\label{product_Laplace}
 L_{\mu_1}L_{\mu_2}u:= (-\Delta -\mu_1)(-\Delta-\mu_2)u=0, \qquad \text{with\ } \beta =\mu_1+\mu_2\ \hbox{and}\  \alpha =-\mu_1\mu_2.
\end{equation}

The eigenvalues $\mu_1$ and $\mu_2$ are also the roots of the polynomial $P(\mu)= \mu^2-\beta \mu-\alpha$,  so that
$$\mu_{1}=\frac{\beta+\sqrt{\beta^2+4\alpha}}{2},\quad\mu_{2}=\frac{\beta-\sqrt{\beta^2+4\alpha}}{2}.$$
 We observe that these kinds of algebraic computations, as well as \eqref{product_Laplace}, are relatively standard for this type of problems; see several examples in the book of Peletier-Troy \cite{PelTro}.

It will be also useful to introduce $\lambda_k$, the eigenvalues of the Laplace operator $(-\Delta)$ under homogeneous Dirichlet boundary conditions in $\Omega$. For such an eigenvalue problem we actually have a family of infinitely many positive ordered eigenvalues, i.e.
$$0<\lambda_1\leq \lambda_2\leq \ldots \leq \lambda_k\leq \ldots$$
associated with a complete family of eigenfunctions $\{\phi_k\}_{k=1}^\infty$.
Now, for the eigenvalue problem \eqref{eigenvalue_problem} we can establish the following result, providing us with an existence condition for non-trivial solutions in terms of the parameters $\alpha$ and $\beta$.
\begin{proposition}
The eigenvalue problem~\eqref{eigenvalue_problem} has (at least) a non-trivial solution if and only if the pair $(\alpha,\beta)$ verifies
\begin{equation}\label{albe_condition}
\frac{\alpha}{\lambda_k^2}+\frac{\beta}{\lambda_k}=1,\quad \hbox{for some $k\in\mathbb{N}$.}
\end{equation}
\end{proposition}

\begin{proof}
First of all observe that, instead of~\eqref{product_Laplace} (or~\eqref{eigenvalue_problem}), we might consider the equivalent pair of Helmholz equations
\begin{equation}
  \label{product_Laplace2} (-\Delta -\mu_2)u=v,\qquad (-\Delta-\mu_1)v=0.\end{equation}
If~\eqref{eigenvalue_problem} has a nontrivial solution, then it is clear that $\mu_1=\lambda_k$ or $\mu_2=\lambda_k$, for some $k\in \mathbb{N}$. Indeed, assume, by contradiction, that $\mu_1\neq\lambda_k$ and $\mu_2\neq\lambda_k$, for all $k\in \mathbb{N}$. In that case,  the second equation in~\eqref{product_Laplace2} has only the trivial solution and consequently, the solution for the first equation is also only the trivial one.
Moreover, if  $\phi_k$ denotes the eigenfunction for $(-\Delta)$ in $\Omega$ under Dirichlet boundary conditions associated with the eigenvalue $\lambda_k$, then, in any case, $\mu_1=\lambda_k$ or $\mu_2=\lambda_k$, the function $u=\phi_k$ is a solution to~\eqref{product_Laplace} and hence to~\eqref{eigenvalue_problem}. So that, finally, substituting $\phi_k$ into the fourth order equation  \eqref{eigenvalue_problem} we arrive at the equality \eqref{albe_condition}.

Conversely, if~\eqref{albe_condition} holds, it is straightforward to see that $u=\phi_k$ is a solution to~\eqref{eigenvalue_problem}.
\end{proof}

{\begin{remark}
  Observe that, as in~\cite{DelPino_Manas}, without further assumptions on $\alpha$ and $\beta$, the coefficients $\mu_1$ and $\mu_2$ may be complex.
\end{remark}}

In addition to the homogeneous problem~\eqref{eigenvalue_problem} consider also the inhomogeneous problem
\begin{equation}\label{eigenvalue_problem_h}
\left\{\begin{array}{ll}
\Delta^2u+\beta \Delta u-\alpha u = h({ x}) & \quad\text{in } \Omega,\\  u=\Delta u=0 & \quad \text{on }\partial \Omega,
\end{array}\right.
\end{equation}
for $h$ positive and continuous in $\Omega$. We will impose that $\alpha$ and $\beta$ verify
{
\begin{equation}
  \label{eq:second.condition.alpha.beta}\beta<2\lambda_1,\quad \text{and}\quad \beta^2\geq -4\alpha.
\end{equation}
These conditions imply that $\mu_1\geq \mu_2>-\lambda_1$. In particular, $\lambda_k+\mu> 0$.}

If we denote by  $G_i({ x},\tau)$, the Green's function of the linear boundary problem
$$(-\Delta-\mu_i)u=0\quad \text{in $\Omega$}$$ together with Dirichlet
homogeneous boundary conditions, then,   using again the expression~\eqref{product_Laplace}, we have that the solution of~\eqref{eigenvalue_problem_h}  is unique and it can be expressed as
\begin{align}\label{2.1}
u({ x})=\int_\Omega \int_\Omega G_{1}({ x},{ \tau})G_{2}({ \tau},{s})  h({s})\, d{s}\, d{ \tau}, \quad { x}\in {{\Omega}}.
\end{align}
Observe that, due to the Fredholm alternative,  the inhomogeneous problem has a unique solution, which is given by the Green's function~\eqref{2.1} if \eqref{albe_condition} fails. This indeed happens if $\mu_i\neq -\lambda_k$ which implies that the operators $(-\Delta-\mu_i)$ do not have the eigenvalue 0 and hence they are invertible.

We now prove some important properties of the Green's function that will play a crucial role in our analysis. Note that in general, for the problem \begin{equation}
\label{second_general}
L u=f\quad \hbox{in}\quad \Omega, \qquad u=0\quad \hbox{on}\quad \partial\Omega
\end{equation}
with $L$  a second order differential operator (self-adjoint to have some extra properties) we find  a solution to~\eqref{second_general} as
\begin{equation}
\label{Green_solution}
u({ x})=L^{-1}[f]({ x})=\int_\Omega G({ x},{ \tau}) f({\tau})d{\tau},
\end{equation}
in terms of the Green's function $G$. Thus, the existence of a Green's function is equivalent to that of a unique solution of~\eqref{second_general}.

\begin{lemma}\label{lemma2.1}
For all ${ x},{ \tau} \in \Omega$,  the Green's function $G$ associated with the differential operator $(-\Delta-\mu)$ {with Dirichlet boundary condition} satisfies the following properties:
\begin{enumerate}
\item[(i)] $G({ x},{  \tau})=G({  \tau},{ x})$ and $G({ x},{  \tau})>0$;
\item[(ii)] $G({ x},{  \tau}) \leq C \sqrt{G({  \tau},{  \tau})}$, where $C>0$ is a constant;
\item [(iii)] $G({ x},{  \tau}) \geq \delta \psi^2({ x})\sqrt{ G({  \tau},{ \tau})}$, where $\delta>0$ is a constant and $\psi$ is the $L_2$-normalized eigenfunction associated with the first eigenvalue, $\lambda_1$, of the linear operator $(-\Delta-\mu)$ with Dirichlet boundary condition.
\end{enumerate}
\end{lemma}

\begin{remark}
  Since $G(x,\tau)=0$ for $x\in\partial \Omega$, we get that  (ii) and (iii) above,  are trivially satisfied for $x\in\partial \Omega$.
\end{remark}

\begin{proof}
First, observe that  $G$ is the Green's function associated with the self-adjoint operator $(-\Delta-\mu)$, whose eigenvalues and eigenfunctions  are just $\lambda_k+\mu$ and $\phi_k$. Moreover, \eqref{eq:second.condition.alpha.beta} implies that $\lambda_k+\mu_i>0$. Therefore, the symmetry property in (i) and the positivity of $G$  are straightforward, see for instance~\cite{Kreith}.

To prove the following items we use the so-called bilinear representation for the Green's functions of Helmholtz type equations (see \cite{Duffy} for further details on the original paper of Sommerfield \cite{Somm}). In particular, we have
\begin{equation}
\label{Green_function}
 G({ x},{ \tau}) =\sum_{k=1}^\infty \frac{\phi_k({ x}) \phi_k(  \tau)}{\lambda_k+\mu}.
\end{equation}
The boundedness of $G$ and the Cauchy-Schwartz inequality applied to  \eqref{Green_function} yields estimate (ii).

Finally, to obtain estimate (iii), it suffices to show that, there exists a constant $\delta$ such that
\begin{equation}
  \label{eq:infimum}
\inf_{x,\tau \in \Omega} \frac{G(x,\tau)}{\psi^2(x)\sqrt{G(\tau,\tau)}}\geq\delta>0.
\end{equation}
To this aim, we observe first that, from (i) we have that $G$ is positive in $\Omega$ and moreover,  $G$ is bounded  in $\overline{\Omega}$ by continuity. Therefore, $\psi^2(x)\sqrt{G(\tau,\tau)}$ is uniformly bounded in $\Omega$. Hence, the infimum in~\eqref{eq:infimum} could only be zero if the numerator approaches zero. However, we will show that this infimum cannot be achieved when the numerator $G(x,\tau)$ approaches zero.

 Indeed, for any $\epsilon>0$, consider the set $A_\epsilon=\{x\in \overline{\Omega}\;|\;0<\frac{\epsilon}{2}<G(x,\tau)<\epsilon, \text{ for all } \tau \in \Omega\}$.
   Since $(-\Delta-\mu)\psi=\lambda_1\psi$, with homogeneous Dirichlet boundary data, using the representation formula, we obtain for any $x\in A_\epsilon$,
\begin{equation}\label{estimate_psi}
    \psi(x)=\lambda_1\int_\Omega G(x,y)\psi(y)dy\leq \epsilon\lambda_1^{1+\frac{n}{4}}|\Omega|,
\end{equation}
where the inequality follows from the estimate $|\psi(y)|\leq \lambda_1^{\frac{n}{4}}$, for all $y\in \Omega$~\cite{Banuelos}. Finally, using~\eqref{estimate_psi}, for any $x\in A_\epsilon$ we obtain
$$
    \frac{G(x,\tau)}{\psi^2(x)\sqrt{G(\tau,\tau)}}\geq \left(2\epsilon\lambda_1^{2+\frac{n}{2}}\max_\Omega \sqrt{G(\tau,\tau)}\right)^{-1}.
$$
Hence, as $\epsilon$ tends to zero,
$$\frac{G(x,\tau)}{\psi^2(x)\sqrt{G(\tau,\tau)}}\to\infty .
$$
Therefore, the infimum~\eqref{eq:infimum} is attained at a point $(x,\tau)$ such that ${G(x,\tau)}$ is positive, and due to the continuity and boundedness of $G$,~\eqref{eq:infimum}, and hence (iii), follows.

\end{proof}

Finally, we include an  estimation for the solution of equation \eqref{eigenvalue_problem_h}. {Its proof is similar to the one performed in \cite{Li} and we omit it here}. Here, and in the sequel  we will denote the maximum norm in $C(\overline{\Omega})$ by  $\|u\|=\max_{\overline{\Omega}} u({ x})$.
\begin{lemma}
\label{lemma_bound_u}
The solution $u$ of the boundary value problem \eqref{eigenvalue_problem_h} satisfies {for all $x\in \Omega$}
\begin{equation}
\label{estimation_Li}
u({x})\geq \frac{\delta_1 \delta_2 C_0 \psi^2_1(x) }{C_1C_2|\Omega|{\sqrt{M_1}}}\|u\|,
\end{equation}
with
$$
C_0=\int_\Omega \sqrt{G_1({ x}, {x})} \psi^2_2({ x}) \,d{x}\text{\quad and\quad}  M_1=\max_{\overline\Omega} G_{1} ({x},{ x}),
$$
where $\delta_i$ and $C_i$, with $i=1,2$ are the quantities $\delta$ and $C$ (respectively) appearing in Lemma~{\rm\ref{lemma2.1}} relatively to Green's function $G_i$, $\psi_i$ are the 
$L_2$-normalized eigenfunctions associated with the first eigenvalue of the linear operator $(-\Delta-\mu_i)$ with Dirichlet boundary condition.
\end{lemma}

\section{Fixed Point Theory}
\label{Section_Existence}


From the previous section it can be seen that we actually can work with integral equations to prove the existence of solutions for problem \eqref{1.1_eq_principal} (based on the Green's functions). In fact, first of all we observe that we can adapt the expression~\eqref{2.1} to the solution of~\eqref{1.1_eq_principal} and write
$$u_j({ x})=\int_\Omega \int_\Omega G_{1,j}({ x},{ \tau})G_{2,j}({ \tau},{ s})  f_{j}({ s}, u_{1}({ s}),u_{2}({s})) d{ s}\, d{ \tau}, \quad { x}\in \overline{\Omega},$$
with $j=1,2$.
In view of this expression,  we shall mainly discuss the existence results for~\eqref{1.1_eq_principal} by using the fixed point index theory. Thus, we define the following mappings:
\begin{equation}\label{3.8}
\begin{split}
    T_{j}(u_1, u_2)({ x})&=\int_\Omega \int_\Omega G_{1,j}({ x},{ \tau})G_{2,j}({ \tau},{ s})  f_{j}({ s}, u_{1}({ s}),u_{2}({s}))  d{ s}\, d{ \tau}, \\
T(u_1, u_2)({ x})&=\left(T_{1}(u_1, u_2)({ x}), T_{2}(u_1, u_2)({ x})\right),
\end{split}
\end{equation}
for all ${ x} \in {\Omega}$, $(u_1, u_2) \in C(\overline{\Omega}) \times C(\overline{\Omega})$. Observe that the existence of non-trivial solutions for system~\eqref{1.1_eq_principal}
is equivalent to the existence of a nontrivial fixed point of $T$. Therefore, we just need to find the nontrivial fixed point of $T$ to establish the existence of non-trivial solutions for~\eqref{1.1_eq_principal}.

We introduce some results of  fixed point index theory which will be play a key role in the subsequent analysis, see \cite{Amann}, \cite{GuoLak} and \cite{Lloyd}.

\begin{definition}
[{\cite[Chapter 1]{GuoLak}}] Let $(E,\|\ \|)$ be a real Banach space. A non-empty, closed, convex set $K\subset E$ is called a \textit{cone} if the following conditions are met:
\begin{enumerate}
    \item[(i)] If ${ v}\in K$, and $a\geq 0$ then $a{ v}\in K$;
    \item[(ii)] If ${ v}\in K$ and $-{ v}\in K$, then ${ v}={ 0}$.
\end{enumerate}
\end{definition}
In our setting we consider the Banach space $E:=C(\overline{\Omega}) \times C(\overline{\Omega})$ endowed with the norm
$${ \|(u_1, u_2)\|=\|u_1\|+\|u_2\|}.$$
Recall that, by abuse of notation, we have denoted the maximum norm of $C(\overline{\Omega})$ also by $\|\cdot\|$.
Having this  in mind we define for the closed subset $\Omega_0$ of $\Omega$,  the set
$$
\mathcal{P}=\{(u_1, u_2) \in E: (u_{1}(x)+u_2(x)) \geq \sigma\|(u_1,u_2)\|,\; \text{ for all } \:{ x} \in{\Omega}_0\}, \quad \text{ with } \sigma=\min\{\sigma_1,\sigma_2\}
$$
where $$\sigma_{j}=\frac{\delta_{1,j}\delta_{2,j}m_{1,j}C_{0,j}}{C_{1,j}C_{2,j}\sqrt{M_{1,j}}|\Omega|}, $$ with
$\delta_{i,j}, C_{0,j}$, $C_{i,j}$, and $M_{1,j}$ as the corresponding quantities coming from Lemma\;\ref{lemma2.1} for $u_j$, and $m_{1,j}=\min_{\Omega_0}\psi^2_{1,j}(x)$, $i,j=1,2$. Note that $0<\sigma_j<\infty$ for all ${{x}}\in \Omega_0$.

\begin{lemma}\label{lemma2.3}
The set $\mathcal{P}$ is a nonempty, convex, and closed subset of $E$. Moreover, $\mathcal{P}$ is a cone of~$E$.
\end{lemma}
\begin{proof}
It is clear that $\mathcal{P}$ is nonempty as ${ 0}\in \mathcal{P}$. Also, thanks to the continuity of $u_1$ and $u_2$ it follows that $\mathcal{P}$ is closed.
To see that it is convex, observe that if $(u_1,u_2), (v_1,v_2)\in \mathcal{P}$ for $t\in(0,1)$
$$t( u_1+u_2)+(1-t)(v_1+v_2)\geq \sigma(t(\|u_1\|+\|u_2\|)+(1-t)(\|v_1\|+\|v_2\|))
\geq\sigma(\|tu_1+(1-t)v_1\|+\|tu_2+(1-t)v_2\|),$$
which implies $t(u_1,u_2)+(1-t)(v_1,v_2)\in\mathcal{P}$.

    Finally we  prove that $\mathcal{P}$ is a cone of the Banach space $E$. Indeed, it is straightforward to check that if $(u_1,u_2)\in \mathcal{P}$ and $a\geq 0$, then  $(a u_1,a u_2)\in \mathcal{P}.$
    Moreover, assume that $(u_1,u_2)$ and $(-u_1,-u_2)$ belong to $\mathcal{P}.$
    Then,
    $
    0\geq 2\sigma (\|u_1\|+\|u_2\|)
    $.
    Since $\sigma>0$ inside $\mathcal{P}$ we conclude that $(u_1,u_2)=(0,0).$
\end{proof}

We now present a lemma that outlines some properties of the mapping $T$.
\begin{lemma}\label{lemma4.1}
The mapping $T: \mathcal{P} \rightarrow \mathcal{P}$, defined by \eqref{3.8}, is completely continuous and  $T(\mathcal{P}) \subset \mathcal{P}$.
\end{lemma}
\begin{proof}
Let $(u_1,u_2)\in \mathcal{P}$. For $j=1,2$, using Lemma~\ref{lemma2.1} we get
\begin{equation}
\begin{split}
T_{j}(u_1, u_2)({ x}) &=\int_\Omega \int_\Omega G_{1,j}({ x},{ \tau})G_{2,j}({ \tau},{ s})  f_{j}({ s}, u_{1}({ s}),u_{2}({ s}))  d{ s}\, d{ \tau}\\
&\leq C_{1,j}C_{2,j}\sqrt{M_{1,j}} |\Omega|\int_\Omega \sqrt{G_{2,j}(s,s)}f_j({ s}, u_{1}({ s}),u_{2}({ s}))ds
\end{split}
\end{equation} for all $x\in {\Omega}$, and hence
\begin{equation}\label{above}
\begin{split}
\|T_{j}(u_1, u_2)\|
&\leq C_{1,j}C_{2,j}\sqrt{M_{1,j}} |\Omega|\int_\Omega \sqrt{G_{2,j}(s,s)}f_j({ s}, u_{1}({ s}),u_{2}({ s}))ds.
\end{split}
\end{equation}
On the other hand, using the estimate \eqref{above} and Lemma \ref{lemma2.1}, we observe for all $x\in {\Omega}$ that
\begin{equation*}
\begin{split}
T_{j}(u_1, u_2)({ x}) &=\int_\Omega \int_\Omega G_{1,j}({ x},{ \tau})G_{2,j}({ \tau},{ s})  f_{j}({ s}, u_{1}({ s}),u_{2}({ s}))  d{ s}\, d{ \tau}\\
&\geq {{\delta_{1,j}\delta_{2,j}}\psi^2_{1,j}(x)}\int_\Omega \int_\Omega \sqrt{G_{1,j}({ \tau},{ \tau})}\psi^2_{2,j}({ \tau})\sqrt{G_{2,j}({ s},{ s})}  f_{j}({ s}, u_{1}({ s}),u_{2}({ s}))  d{ s}\, d{ \tau}\\
&\geq {{\delta_{1,j}\delta_{2,j}}\psi^2_{1,j}(x)}\int_\Omega  \sqrt{G_{1,j}({ \tau},{ \tau})}\psi^2_{2,j}({ \tau})d{\tau}\int_\Omega \sqrt{G_{2,j}({ s},{ s})}  f_{j}({ s}, u_{1}({ s}),u_{2}({ s}))  d{ s}\\
&={{\delta_{1,j}\delta_{2,j}}\psi^2_{1,j}(x)}{{C_{0,j}}}\int_\Omega \sqrt{G_{2,j}({ s},{ s})}  f_{j}({ s}, u_{1}({ s}),u_{2}({ s}))  d{ s}\\
&\geq \frac{{{\delta_{1,j}\delta_{2,j}}\psi^2_{1,j}(x)}{{C_{0,j}}}}{C_{1,j}C_{2,j}\sqrt{M_{1,j}}|\Omega|}\|T_{j}(u_1, u_2)\|,\qquad j=1,2.
\end{split}
\end{equation*} If $x\in\Omega_0$, we get that $
T_{j}(u_1, u_2)({ x})\geq \sigma_j \|T_{j}(u_1, u_2)\|
$,
which yields
$$
T_1(u_1,u_2)(x)+T_2(u_1,u_2)(x)\geq \sigma_1\|T_1(u_1,u_2)\|+\sigma_2\|T_2(u_1,u_2)\|\geq \sigma\|T(u_1,u_2)\|
$$
and hence
 $T(u_1, u_2)=(T_1(u_1, u_2),T_2(u_1, u_2)) \in \mathcal{P}$, that is, $T(\mathcal{P}) \subset \mathcal{P}$.

{In addition, note that $f_{1}, f_{2}$, and $G_{i, j}$ are continuous. Therefore, we can deduce that $T$ is completely continuous just applying Ascoli-Arzela Theorem.}
\end{proof}

Finally, let us denote
$$
\mathcal{P}_{r}:=\{(u_1, u_2) \in \mathcal{P}:\|(u_1, u_2)\|<r\}.
$$
Clearly, for each $r>0, \mathcal{P}_{r}$ is a relatively open and bounded set of $\mathcal{P}$.
Then, by the definition of cone $\mathcal{P}$ and the norm $\|(u_1, u_2)\|$, one can see that
$$
\partial \mathcal{P}_{r}:=\{(u_1, u_2) \in \mathcal{P}:\|(u_1, u_2)\|=r\},
\quad \overline{\mathcal{P}_{r}}:=\{(u_1, u_2) \in \mathcal{P}:\|(u_1, u_2)\| \leq r\}.
$$
Since $\mathcal{P}_r\neq \emptyset$ and $T:\overline{\mathcal{P}_{r}} \ \rightarrow \mathcal{P}$ is a completely continuous mapping, see Lemma\;\ref{lemma2.3}, we get that the fixed point index, $i\left(T, \mathcal{P} _r, \mathcal{P} \right)$ is defined if $T(u_1,u_2)\neq (u_1,u_2)$ for ever $(u_1,u_2)\in \mathcal{P}_r$. Moreover, if $i\left(T, \mathcal{P} _r, \mathcal{P} \right)\neq 0$, this actually implies that the mapping $T$ possesses a fixed point in $\mathcal{P} _r$.
Recall that the fixed point index $i$ is a counter of the number of zeros for a continuous differential operator. It might be obtained using the Leray-Schauder formula of the topological degree for compact perturbations of the identity in a Banach space and it is related to the topological degree of Brouwer. Through these abstract algebraic concepts one can obtain the number of solutions of a differential equation; see \cite{Amann}, \cite{GuoLak} and \cite{Lloyd} for  further details on fixed point theory an
extensive analysis of such concepts.

The following lemma states how to calculate the fixed point index of $T$ in $\mathcal{P}_{r}$, $i\left(T, \mathcal{P}_r, \mathcal{P}\right)$. We include it here omitting the proof, which can be checked in  \cite{GuoLak}.

\begin{lemma}\label{Li2.4}
Let $T: \mathcal{P} \rightarrow \mathcal{P}$ be a completely continuous mapping.
\begin{enumerate}
\item If $\eta T (u_1,u_2) \neq (u_1,u_2)$ for every $(u_1,u_2) \in \partial \mathcal{P} _r$ and $0<\eta\leq 1$, then $i\left(T, \mathcal{P} _r, \mathcal{P} \right)=1$.
\item If  $\eta T (u_1,u_2) \neq (u_1,u_2)$ for every $(u_1,u_2) \in \partial \mathcal{P} _r$, $\eta \geq 1$  and {$\inf _{(u_1,u_2) \in \partial \mathcal{P} _r}\|T(u_1,u_2)\|>0$} then $i\left(T, \mathcal{P}_r, \mathcal{P}\right)=0$.
\end{enumerate}
\end{lemma}

\section{Existence of positive solutions}
\label{proof_theorem}

In this section we proof the main theorem of the paper. As mentioned, we will establish conditions on the functions $f_i$ and on the parameters of the system~\eqref{1.1_eq_principal} so that~\eqref{1.1_eq_principal} has positive solution $(u_1,u_2)$. Recall that the functions $f_1,f_2   \,:\, \Omega\times [0,\infty)\times [0,\infty) \rightarrow [0,\infty)$ are continuous.

More precisely,
we will assume that the parameters of the system verify
\begin{equation}
  \label{eq:condition.parameter}
   \beta_i<2\lambda_1,\quad  \beta_i^2\geq -4\alpha_i,\quad \frac{\alpha_i}{\lambda_k^2}+\frac{\beta_i}{\lambda_k}<1.
\end{equation}
Recall, see Section~\ref{sect:preliminaries_elements}, that the first two conditions are related to the existence and properties of the Green's functions and the third one to the existence of solutions.
As for the functions $f_i$, let us  introduce the following notation for convenience:
\begin{align*}
 {f}_{0}&=\liminf_{u+u_2\rightarrow 0^+}\min_{x\in\overline{\Omega} }F[x,u_1,u_2], &  {f}_{\infty}&=\liminf_{u_1+u_2\rightarrow +\infty}\min_{x\in\overline{\Omega}} F[x,u_1,u_2], \\
  {f}^{0}&=\limsup_{u_1+u_2\rightarrow 0^+}\max_{x\in\overline{\Omega}}F[x,u_1,u_2], &  {f}^{\infty}&=\limsup_{u_1+u_2\rightarrow +\infty}\max_{x\in\overline{\Omega}} F[x,u_1,u_2],
\end{align*}
where
$$
F[x,u_1,u_2]:=\frac{f_1(x,u_1,u_2)+f_2(x,u_1,u_2)}{L_1u_1+L_2u_2}
$$ and
$L_i=\lambda_1^2-\lambda_1\beta_i-\alpha_i$ for $i=1,2$. Observe that, due to the third condition in~\eqref{eq:condition.parameter}, $L_i>0$.

\begin{theorem}
\label{thm:main}
  If $f^0 < 1< f_\infty $, then the system \eqref{1.1_eq_principal}--\eqref{1.1_BC} has at least one positive solution.
\end{theorem}

\begin{proof}
First, since ${f}^{0}<1$, there exists $\varepsilon\in(0,1)$ and $R_0>0$, small, so that
\begin{equation}
  \label{eq:proof.bound.above}
  f_1(x,u_1,u_2)+f_2(x,u_1,u_2)\leq (1-\varepsilon)(L_1u_1+L_2u_2),
\end{equation}
for all $x\in\Omega$ and $u_1,u_2\geq 0$, $u_1+u_2\leq R_0$.
We claim now that
 for every $(u_1,u_2) \in \partial \mathcal{P} _{R_0}$ and $0<\eta\leq 1$
$$\eta T (u_1,u_2) \neq (u_1,u_2).$$
Then, following Lemma~\ref{Li2.4}, we conclude that
\begin{equation}
  \label{eq:index1} i\left(T, \mathcal{P} _{R_0}, \mathcal{P} \right)=1.
\end{equation}

To proof the claim, we argue by contradiction by assuming that there exist $(u_1^0,u_2^0)\in \partial \mathcal{P}_{R_0}$ and $0<\eta_0 \leq 1$ such that $\eta_0 T (u_1^0,u_2^0) = (u_1^0,u_2^0)$. Then, by definition of
$T$, we have that $(u_1^0,u_2^0)$
satisfies differential equations
\begin{align}\label{Li7}
\Delta^2 u^0_1 +\beta_1 \Delta u^0_1-\alpha_1 u^0_1=\eta_0 f_1(x,u^0_1,u^0_2),\quad \Delta^2 u^0_2 +\beta_2 \Delta u^0_2-\alpha_2 u^0_2=\eta_0 f_2(x,u^0_1,u^0_2)
\end{align}
and boundary condition \eqref{1.1_BC}. Adding these equations and using~\eqref{eq:proof.bound.above} we get
\begin{equation}
  \label{proof.verify.eqs}
\Delta^2 u^0_1 +\beta_1 \Delta u^0_1-\alpha_1 u^0_1+  \Delta^2 u^0_2 +\beta_2 \Delta u^0_2-\alpha_2 u^0_2\leq (1-\varepsilon)(L_1u_1+L_2u_2).
\end{equation}
 Multiplying this expression by $\phi_1(x)$ and integrating by parts we have
 $$
 \int_\Omega (L_1u_1+L_2u_2)\phi_1 \leq (1-\varepsilon) \int_\Omega (L_1u_u+L_2u_2)\phi_1,
 $$
which is a contradiction, since  $(L_1u_1+L_2u_2)\phi_1\geq 0$ in $\Omega$, and hence~\eqref{eq:index1} follows.

On the other hand, we have that due to $ f_\infty>1$, there exists $\varepsilon\in(0,1)$ and $k>0$, so that
\begin{equation}
  \label{eq:proof.bound.below2} f_1(x,u_1,u_2)+f_2(x,u_1,u_2)\geq (1+\varepsilon)(L_1u_1+L_2u_2),\end{equation}
for all $x\in\Omega$ and $u_1,u_2\geq 0$, $u_1+u_2\geq k$.
{Moreover if we take $C=k(1+\varepsilon)(L_1+L_2)$ then for all $x\in\Omega$ and $u_1+u_2\geq 0$
\begin{equation}
  \label{eq:proof.bound.below}
  f_1(x,u_1,u_2)+f_2(x,u_1,u_2)\geq (1+\varepsilon)(L_1u_1+L_2u_2)-C.\end{equation}
}
We want to show now that there exists an $R>R_0$, to be chosen later, so that
$$\inf_{(u_1,u_2)\in \mathcal{P}_R}\|T(u_1,u_2)\|>0,\quad \text{and}\quad \eta T (u_1,u_2) \neq (u_1,u_2)$$
for every $(u_1,u_2) \in \partial \mathcal{P} _R$ and $\eta\geq 1$. As before, this implies, following Lemma~\ref{Li2.4}, that
  \begin{equation}
  \label{eq:index2} i\left(T, \mathcal{P} _R, \mathcal{P} \right)=0.
\end{equation}

We argue again by contradiction. Let $(u_1^0,u_2^0)\in \partial \mathcal{P}_R$ and $\eta_0 \geq 1$ such that $\eta_0 T (u_1^0,u_2^0) = (u_1^0,u_2^0)$. Then, from~\eqref{Li7}  and following the same steps as in the first part of the proof, using~\eqref{eq:proof.bound.below}, we get
$$
 \int_\Omega (L_1u_1+L_2u_2)\phi_1 \geq (1+\varepsilon) \int_\Omega (L_1u_u+L_2u_2)\phi_1-C\int_\Omega \phi_1.
$$
Consequently, we obtain that
\begin{align}\label{Li12}
\varepsilon\int_\Omega (L_1u_1^0+L_2u_2^0) \phi_1 \leq {C\displaystyle\int_\Omega \phi_1 }.
\end{align}
On the other hand, since $(u_1^0,u_2^0)\in\mathcal{P}_R$, we have
 $$
\int_\Omega (L_1u_1^0+L_2u_2^0)\geq \min\{L_1,L_2\}\sigma \|(u_1^0,u_2^0)\|\int_\Omega \phi_1,
 $$
 so that we conclude that
 $$
 \|(u_1^0,u_2^0)\|\leq \frac{C}{\sigma\varepsilon\min\{L_1,L_2\}}:=R_1.
 $$
 If $R>R_1$ this last inequality provides a contradiction with the fact that $(u_1^0,u_2^0)\in \partial \mathcal{P}_R$.

Next, we show that $\inf_{(u_1,u_2)\in \mathcal{P}_R}\|T(u_1,u_2)\|>0$. To this aim, let
$R_2=k/\sigma$ and for $R>R_2$ take $(u_1,u_2) \in \partial \mathcal{P}_R$.
Then, by definition of the cone $\mathcal{P}$,
for $x \in {\Omega}_0$, we have that $(u_1+u_2)(x) \geq \sigma\|(u_1,u_2)\|>k$ . So that,  due to~\eqref{eq:proof.bound.below2} we find that
\begin{equation}
  \begin{aligned}
  \label{eq:bound.for.f}
f_1(x,u_1,u_2)+f_2(x,u_1,u_2)&\geq (1+\varepsilon)(L_1u_1+L_2u_2)\\ & \geq  (1+\varepsilon)\min\{L_1,L_2\}\sigma\|(u_1,u_2)\|\geq (1+\varepsilon)\min\{L_1,L_2\}k,
\end{aligned}
\end{equation}
for all $x\in\Omega_0$.
On the other hand,  let $x_0\in\Omega_0$ fixed, as in the proof of Lemma~\ref{lemma4.1}, using Lemma~\ref{lemma2.1}, we get
$$
\begin{aligned}
(T _1  (u_1,u_2)&  +T_2(u_1,u_2))(x_0)  \geq 
{{\delta_{1,1}\delta_{2,1}}\psi^2_{1,1}}{{C_{0,1}}}\int_\Omega \sqrt{G_{2,1}({ s},{ s})}  f_{1}({ s}, u_{1}({ s}),u_{2}({ s}))  d{ s}
\\ & +
{{\delta_{1,2}\delta_{2,2}}\psi^2_{1,2}}{{C_{0,2}}}\int_\Omega \sqrt{G_{2,2}({ s},{ s})}  f_{2}({ s}, u_{1}({ s}),u_{2}({ s}))  d{ s}\\
\geq &
\min_{j=1,2}\{\delta_{1,j}\delta_{2,j}m_{1,j}C_{0,j}\}\int_{\Omega_0} \left(\sqrt{G_{2,1}({ s},{ s})}  f_{1}({ s}, u_{1}({ s}),u_{2}({ s})) + \sqrt{G_{2,2}({ s},{ s}) } f_{2}({ s}, u_{1}({ s}),u_{2}({ s})) \right) d{ s}\\
\geq&\min_{j=1,2}\{\delta_{1,j}\delta_{2,j}m_{1,j}m_{2,j}C_{0,j}\} \int_{\Omega_0} f_{1}({ s}, u_{1}({ s}),u_{2}({ s}))+f_{2}({ s}, u_{1}({ s}),u_{2}({ s}))  d{ s},
\end{aligned}
$$
where $m_{2,j}=\min_{\Omega_0}G_{2,j}(s,s)$.
From~\eqref{eq:bound.for.f} we have that
\begin{equation}
  \label{eq:bound.inf.T}
\int_{\Omega_0} f_{1}({ s}, u_{1}({ s}),u_{2}({ s}))+f_{2}({ s}, u_{1}({ s}),u_{2}({ s}))  d{ s} \geq (1+\varepsilon)\min\{L_1,L_2\}k |\Omega_0|,
\end{equation}
so that we conclude that
$$
(T _1(u_1,u_2)+T_2(u_1,u_2))(x_0)\geq (1+\varepsilon)\min_{j=1,2}\{\delta_{1,j}\delta_{2,j}m_{1,j}m_{2,j}C_{0,j}\}\min\{L_1,L_2\}k |\Omega_0|.
$$
Therefore
{$$\begin{aligned}
\|T(u_1,u_2)\| \geq  (T_1 (u_1,u_2)+T_2(u_1,u_2)) (x_0) \geq (1+\varepsilon)\min_{j=1,2}\{\delta_{1,j}\delta_{2,j}m_{1,j}m_{2,j}C_{0,j}\}\min\{L_1,L_2\}k |\Omega_0|.
\end{aligned}$$}
Taking the infimum on both sides over $(u_1,u_2)\in \partial \mathcal{P}_R$ we obtain that $\inf_{(u_1,u_2)\in \partial \mathcal{P}_R}\|T(u_1,u_2)\|>0$.

Summing up, for any $R_0>0$ small and $R>\max \left\{R_0,R_1, R_2\right\}$ we conclude the proof by using the additivity of fixed point index for~\eqref{eq:index1} and~\eqref{eq:index2}:
$$
i\left(T, \mathcal{P}_R \backslash \overline{P}_{R_0}, \mathcal{P}\right)=i\left(T, \mathcal{P}_R, \mathcal{P}\right)-i\left(T, \mathcal{P}_{R_0}, P\right)=-1 .
$$
{Hence, since $R_0$ is as small as desired, $T$ has a fixed point in $\mathcal{P}_R \backslash \{0,0\}$, which is a positive solution of the system \eqref{1.1_eq_principal}--\eqref{1.1_BC}}.
\end{proof}


\begin{theorem}
  If $ f^\infty <1<f_0$, then the system \eqref{1.1_eq_principal}--\eqref{1.1_BC} has at least one positive solution.
\end{theorem}

The proof follows the same arguments as the previous one. We omit the details and sketch the main differences.

\begin{proof}
  Since $1<f_0$, there exists $\varepsilon\in(0,1)$ and $R_0>0$, small, so that
$$
 f_1(x,u_1,u_2)+f_2(x,u_1,u_2)\geq (1+\varepsilon)(L_1u_1+L_2u_2),
$$
for all $x\in\Omega$ and $u_1,u_2\geq 0$, $u_1+u_2\leq R_0$.
Then for every $(u_1,u_2) \in \partial \mathcal{P}_{R_0}$, through the argument used in \eqref{eq:bound.inf.T}, we have
$$
\|T (u_1,u_2)\| \geq (1+\varepsilon)\min_{j=1,2}\{\delta_{1,j}\delta_{2,j}m_{1,j}m_{2,j}C_{0,j}\}\min\{L_1,L_2\}R_0 |\Omega_0|.
$$
Hence $\inf _{(u_1,u_2) \in \partial \mathcal{P}_{R_0}}\|T (u_1,u_2)|>0$.

To show that $\eta T (u_1,u_2) \neq (u_1,u_2)$ for any $(u_1,u_2) \in \partial \mathcal{P}_{R_0}$ and $\eta \geq 1$ we argue by contradiction and get that, if there exist $(u_1^0,u_2^0) \in \partial \mathcal{P}_{R_0}$ and $\eta_0 \geq 1$ such that $\eta_0 T (u_1^0,u_2^0)=(u_1^0,u_2^0)$, then
$$
 \int_\Omega (L_1u_1+L_2u_2)\phi_1 \geq (1+\varepsilon) \int_\Omega (L_1u_u+L_2u_2)\phi_1,
$$
which implies $1\geq (1+\epsilon)$. Hence, we conclude that $i\left(T, \mathcal{P}_{R_0}, \mathcal{P}\right)=0$.

On the other hand, $f^\infty<1$ implies that there exist there exists $\varepsilon\in(0,1)$ and $k>0$, so
$$
   f_1(x,u_1,u_2)+f_2(x,u_1,u_2)\leq (1-\varepsilon)(L_1u_1+L_2u_2),$$
for all $x\in\Omega$ and $u_1,u_2\geq 0$, $u_1+u_2\geq k$. Moreover, we can find $C>0$ so that, for all $x\in\Omega$ and $u_1+u_2\geq 0$
$$
  f_1(x,u_1,u_2)+f_2(x,u_1,u_2)\leq (1-\varepsilon)(L_1u_1+L_2u_2)+C.$$

We argue as in the proof of Theorem~\ref{thm:main} assuming that for $R>R_0$ there exist {$(u_1^0,u_2^0 )\in \partial \mathcal{P}_R$} and $0<\eta_0 \leq 1$ such that $\eta_0 T(u_1^0,u_2^0)=(u_1^0,u_2^0)$. We get that
 $$
 \|(u_1^0,u_2^0)\|\leq \frac{C}{\sigma\varepsilon\min\{L_1,L_2\}}:=R_1,
 $$
which is a contradiction if  $R>R_1$. Hence, if  $R>\max \{{R}_0,{R}_1\}$ we get $i\left(T, \mathcal{P}_R, \mathcal{P}\right)=1$.

Finally it follows that
$i\left(T, \mathcal{P}_R \backslash \overline{\mathcal{P}}_{R_0}, \mathcal{P}\right)=i\left(T, \mathcal{P}_R, \mathcal{P}\right)-i\left(T, \mathcal{P}_{R_0}, \mathcal{P}\right)=1$.
Consequently, $T$ has a fixed point in $\mathcal{P}_R \backslash \overline{\mathcal{P}}_{R_0}$, which is a {positive} solution of \eqref{1.1_eq_principal}--\eqref{1.1_BC}.
\end{proof}



\end{document}